\documentclass[11pt]{amsart}
\usepackage{amsmath,amsthm,amsfonts,amscd,amssymb,eucal,latexsym,mathrsfs}

\newtheorem{theorem}{Theorem}[section]

\newtheorem{lemma}[theorem]{Lemma}

\theoremstyle{definition}
\newtheorem{definition}[theorem]{Definition}

\newcommand{\cB}{{\mathcal B}}

\newcommand{\cF}{{\mathcal F}}

\newcommand{\cT}{{\mathcal T}}

\newcommand{\cP}{{\mathcal P}}

\newcommand{\cS}{{\mathcal S}}

\newcommand{\Eb}{{\mathbb E}}

\newcommand{\Pb}{{\mathbb P}}

\newcommand{\Nb}{{\mathbb N}}

\newcommand{\Sym}{{\rm Sym}}

\newcommand{\Hom}{{\rm Hom}}

\newcommand{\eps}{\varepsilon}
\newcommand{\Var}{{\rm Var}}
\newcommand{\Prob}{\Pb}
\newcommand{\Map}{{\rm Map}}

\allowdisplaybreaks

\begin{document}

\title{Sofic measure entropy via finite partitions}

\author{David Kerr}
\address{\hskip-\parindent
David Kerr, Department of Mathematics, Texas A{\&}M University,
College Station TX 77843-3368, U.S.A.}
\email{kerr@math.tamu.edu}

\begin{abstract}
We give a generator-free formulation of sofic measure entropy
using finite partitions and establish a Kolmogorov-Sinai theorem.
We also show how to compute the values 
for general Bernoulli actions in a concise way using the 
arguments of Bowen in the finite base case.
\end{abstract}

\date{November 5, 2011}

\maketitle

\section{Introduction}

Kolmogorov introduced the notion of entropy for a measure-preserving transformation
of a probability space by first defining it locally 
on partitions as a limit of averages of the Shannon entropy under iteration
and then showing that all generating partitions with finite Shannon entropy yield the same value. 
To handle the case in which there is no such generating partition, Sinai proposed taking 
the supremum over all finite partitions, thereby furnishing what is now the standard definition.
The Kolmogorov-Sinai theorem asserts that this supremum coincides
with Kolmogorov's entropy in the presence of a generator.
Ultimately it was realized that this set-up works 
most generally for measure-preserving actions of amenable groups,
and in the case of a countably infinite amenable group Ornstein and Weiss showed that entropy
is a complete invariant for Bernoulli actions \cite{OrnWei87}.

Recently Lewis Bowen greatly extended the scope of this classical theory 
to measure-preserving actions of countable sofic groups by replacing the internal 
information-theoretic approach of Kolmogorov with 
the statistical-mechanical idea of counting external finite models \cite{Bow10,Bow11a}.
Bowen first defined the dynamical entropy of
a finite partition by measuring the exponential growth of
the number of models with respect to a fixed sofic approximation sequence for the group,
and then he showed that any two generating finite partitions produce the same value. 
By a limiting process he also extended this analysis to partitions with finite Shannon entropy.
Bowen was thereby able to extend the Ornstein-Weiss entropy classification of Bernoulli actions
to a wide class of nonamenable acting groups, including all nontorsion countably infinite
sofic groups.

For various reasons, including the possibility of formulating a variational principle,
one would like to remove the generator assumption in Bowen's definition.
Sinai's solution of taking a supremum over finite partitions does not work in this case, 
as illustrated by the fact that any two nontrivial Bernoulli actions of a countable group containing 
the free group $F_2$ factor onto one another \cite{Bow11}. 
To circumvent this problem, Hanfeng Li and the author developed an operator algebra
approach using approximate homomorphisms that allows
one to broaden the meaning of generator to bounded sequences of functions in $L^\infty$.
This leads to a completely general notion of sofic measure entropy, 
as well as a topological counterpart and a variational principle relating the two \cite{KerLi11}. 
As in Bowen's setting, when the acting group is amenable one 
recovers the Kolmogorov-Sinai entropy \cite{KerLi11a}.

The question still remained, however, whether there exists a generator-free definition
that uses only finite partitions, in the spirit of Sinai. The aim of this note is
to provide such a definition and to establish a Kolmogorov-Sinai theorem for it,
which we do in Section~\ref{S-definition}.
We continue to use the homomorphism perspective of \cite{KerLi11},
although we avoid linearizing and work only
with algebras of sets. The novelty here is to define the entropy locally 
with respect to two partitions playing different roles. The sofic modeling of the
dynamics is expressed with respect to one partition, as in Lewis Bowen's original approach,
while the second partition is used to express the observational scale at which we are 
able to distinguish different models, in analogy with Rufus Bowen's $(n,\eps)$-separated set
definition of topological entropy for homeomorphisms. We take an infimum over all partitions 
playing the first role while holding the second one fixed, and then take a supremum 
over all partitions playing the second role.
We then prove a Kolmogorov-Sinai theorem (Theorem~\ref{T-generator}) 
which says that to compute the entropy it suffices
that all of these partitions range within a given generating $\sigma$-algebra.
This provides for a relatively concise computation of the entropy of general Bernoulli actions of sofic groups
which builds on arguments from \cite{Bow10}. We provide complete details 
of this computation in Section~\ref{S-Bernoulli}. In Section~\ref{S-comparison} we show
that our definition of sofic measure entropy is equivalent to the one from \cite{KerLi11} and hence also,
in the presence of a generating partition with finite Shannon entropy, to Lewis Bowen's definition.
\medskip

\noindent{\it Acknowledgements.}
This work was partially supported by NSF grant DMS-0900938. 
I thank Lewis Bowen and Hanfeng Li for corrections.

\section{Definitions and a Kolmogorov-Sinai theorem}\label{S-definition}

Let $G$ be a countable sofic group. Soficity means that there are a
sequence $\{ d_i \}_{i=1}^\infty$ of positive integers 
and a sequence $\{ \sigma_i : G \to \Sym (d_i ) \}_{i=1}^\infty$
which is asymptotically multiplicative and free in the sense that
\[
\lim_{i\to\infty}
\frac{1}{d_i} \big| \{ k\in \{ 1,\dots ,d_i \} : \sigma_{i,st} (k) = \sigma_{i,s} \sigma_{i,t} (k) \} \big| = 1
\]
for all $s,t\in G$ and
\[
\lim_{i\to\infty}
\frac{1}{d_i} \big| \{ k\in \{ 1,\dots ,d_i \} : \sigma_{i,s} (k) \neq \sigma_{i,t} (k) \} \big| = 1 
\]
for all distinct $s,t\in G$.
Throughout the paper $\Sigma = \{ \sigma_i : G \to \Sym (d_i ) \}_{i=1}^\infty$ will be
a fixed but arbitrary such sofic approximation sequence.

Let $(X,\cB ,\mu)$ be a probability space and $G\curvearrowright X$
a measure-preserving action, which we fix for the purposes of this section and the next.

Write $\cP_d$ for the power set of $\{ 1,\dots ,d \}$, which we view as a $\sigma$-algebra.
The uniform probability measure on $\{ 1,\dots , d \}$ will always be written $\zeta$.
For a family $\cF$ of subsets of
a set $X$ we write $\Sigma (\cF )$ for the $\sigma$-algebra generated by $\cF$.
For a measurable partition $\alpha$ and a finite set $F\subseteq G$ we write $\alpha_F$
for the partition $\{ \bigcap_{s\in F} sA_s : A\in\alpha^F \}$ where $A_s$ denotes the value of $A$ at $s$.

We write $N_\eps (\cdot ,\rho)$ for the maximal cardinality of an $(\eps,\rho)$-separated subset,
i.e., a subset which is $\eps$-separated with respect to $\rho$.

\begin{definition}\label{D-hom}
Let $\alpha$ be a finite measurable partition of $X$, $F$ a finite subset of $G$, and $\delta > 0$.
Let $\sigma$ be a map from $G$ to $\Sym(d)$ for some $d\in\Nb$.
Define $\Hom_\mu (\alpha ,F,\delta ,\sigma )$ to be the set of all homomorphisms
$\varphi : \Sigma (\alpha_F )\to \cP_d$ such that
\begin{enumerate}
\item[(i)] $\sum_{A\in\alpha} |\sigma_s \varphi (A) \Delta \varphi (sA)|/d < \delta$ for all $s\in F$, and

\item[(ii)] $\sum_{A\in\alpha_F} | \zeta (\varphi (A)) - \mu (A) | < \delta$.
\end{enumerate}
\end{definition}

For a partition $\xi\leq\alpha$, write $|\Hom_\mu (\alpha,F,\delta,\sigma)|_\xi$
for the cardinality of the set of restrictions of elements of $\Hom_\mu (\alpha,F,\delta,\sigma)$ to 
$\xi$. Using sums in (i) and (ii) of Definition~\ref{D-hom} ensures the
convenient monotonicity property
that $\Hom_\mu (\alpha,F,\delta,\sigma)\supseteq \Hom_\mu (\alpha',F',\delta',\sigma)$
and hence $|\Hom_\mu (\alpha,F,\delta,\sigma)|_\xi \geq |\Hom_\mu (\alpha',F',\delta',\sigma)|_{\xi'}$
whenever $\alpha\leq\alpha'$, $F\subseteq F'$, $\delta\geq\delta'$, and $\xi\geq\xi'$.

\begin{definition}\label{D-sofic meas ent}
Let $\cS$ be a $\sigma$-subalgebra of $\cB$.
Let $\xi$ and $\alpha$ be finite measurable partitions of $X$ with $\alpha\geq\xi$.
Let $F$ be a nonempty finite subset of $G$ and $\delta > 0$. We define
\begin{align*}
h_{\Sigma ,\mu}^\xi (\alpha,F,\delta) &=
\limsup_{i\to\infty} \frac{1}{d_i} \log |\Hom_\mu (\alpha,F,\delta,\sigma_i)|_\xi ,\\
h_{\Sigma ,\mu}^\xi (\alpha,F) &= \inf_{\delta > 0} h_{\Sigma ,\mu}^\xi (\alpha,F,\delta) ,\\
h_{\Sigma ,\mu}^\xi (\alpha) &= \inf_F h_{\Sigma ,\mu}^\xi (\alpha,F) ,\\
h_{\Sigma ,\mu}^\xi (\cS) &= \inf_\alpha h_{\Sigma ,\mu}^\xi (\alpha) , \\
h_{\Sigma ,\mu} (\cS) &= \sup_\xi h_{\Sigma ,\mu}^\xi (\cS)
\end{align*}
where the infimum in the third line is over all nonempty finite subsets of $G$,
the infimum in the fourth line is over all finite partitions $\alpha\subseteq\cS$ which refine $\xi$, and 
the supremum in the last line is over all finite partitions in $\cS$.
If $\Hom_\mu (\alpha,F,\delta,\sigma_i)$ is empty for all sufficiently large $i$,
we set $h_{\Sigma ,\mu}^\xi (\alpha,F,\delta) = -\infty$.
\end{definition}

\begin{definition}
The measure entropy $h_{\Sigma ,\mu} (X,G)$ of the action with respect to $\Sigma$ is defined 
as $h_{\Sigma ,\mu} (\cB)$.
\end{definition}

We now aim to establish in Theorem~\ref{T-generator} the analogue of the Kolmogorov-Sinai theorem
in our context.

\begin{lemma}\label{L-hom}
Let $\alpha$ be a finite measurable partition of $X$ and let $\eps > 0$. 
Then there is a $\delta > 0$ such that for every $\sigma$-subalgebra $\cS$ of $\cB$ with 
$\max_{A\in\alpha} \inf_{B\in\cS} \mu(A\Delta B) < \delta$
there exists a homomorphism $\theta : \Sigma(\alpha) \to \cS$ satisfying $\mu (\theta(A)\Delta A) < \eps$
for all $A\in\Sigma(\alpha)$.
\end{lemma}

\begin{proof}
Let $\delta > 0$, and let $\cS$ be a $\sigma$-subalgebra of $\cB$
such that $\max_{A\in\alpha} \inf_{B\in\cS} \mu(A\Delta B) < \delta$.
Let $A_1 ,\dots ,A_n$ be an enumeration of the elements of $\alpha$. 
For $i=1,\dots,n-1$ we recursively define $\theta(A_i)$ 
to be an element of $\cS$ contained in the complement of $\theta(A_1)\cup\dots\cup\theta(A_{i-1})$ so that
the quantity $\mu(\theta(A_i) \Delta A_i)$ is minimized. 
Then define $\theta(A_n)$ to be the complement of $\theta(A_1)\cup\dots\cup\theta(A_{n-1})$,
which gives us a homomorphism $\theta : \Sigma(\alpha) \to \cS$.
It is then readily seen that if $\delta$ is small enough as a function of $\eps$ and $|\alpha|$ 
we will have $\mu (\theta(A)\Delta A) < \eps$ for all $A\in\Sigma(\alpha)$.
\end{proof}

Define on the set of all homomorphisms from some $\sigma$-subalgebra of $\cB$ containing $\xi$
to $\cP_d$ the pseudometric
\[
\rho_\xi (\varphi,\psi) = \max_{A\in\xi} \frac1d |\varphi(A)\Delta\psi(A)| .
\]
For $\eps > 0$ set
\begin{align*}
h_{\Sigma ,\mu}^{\xi,\eps} (\alpha,F,\delta) &=
\limsup_{i\to\infty} \frac{1}{d_i} \log N_\eps(\Hom_\mu (\alpha,F,\delta,\sigma_i),\rho_\xi) ,\\
h_{\Sigma ,\mu}^{\xi,\eps} (\alpha,F) &= \inf_{\delta > 0} h_{\Sigma ,\mu}^\xi (\alpha,F,\delta) ,\\
h_{\Sigma ,\mu}^{\xi,\eps} (\alpha) &= \inf_F h_{\Sigma ,\mu}^\xi (\alpha,F) 
\end{align*}
where the last infimum is over the nonempty finite subsets of $G$.

\begin{lemma}\label{L-eps entropy}
Let $\gamma$ be a finite measurable partition and let $\kappa > 0$. Then there is an $\eps > 0$
such that $h_{\Sigma ,\mu}^\gamma (\beta) \leq h_{\Sigma ,\mu}^{\gamma,\eps} (\beta) + \kappa$
for all finite measurable partitions $\beta$ refining $\gamma$.
\end{lemma}

\begin{proof}
This follows from the fact that, given $\eps > 0$, $A\subseteq \{ 1,\dots ,d\}$, and $d\in\Nb$, the set
of all $B\subseteq \{ 1,\dots ,d\}$ such that $\zeta (B\Delta A) < \eps$ has cardinality at most 
$\binom{d}{\lfloor\eps d\rfloor}$, which by Stirling's approximation
is less than $e^{\kappa d}$ for some $\kappa > 0$ depending on $\eps$ but not on $d$ with
$\kappa\to 0$ as $\eps\to 0$.
\end{proof}

\begin{theorem}\label{T-generator}
Let $\cS$ be a generating $\sigma$-subalgebra of $\cB$. Then
$h_{\Sigma ,\mu} (X,G) = h_{\Sigma ,\mu} (\cS)$.
\end{theorem}

\begin{proof}
By symmetry it suffices to show that if $\cT$ is another generating $\sigma$-subalgebra of $\cB$ then
$h_{\Sigma ,\mu} (\cT) \leq h_{\Sigma ,\mu} (\cS)$. 

Let $\gamma$ be a finite partition in $\cT$. By Lemma~\ref{L-eps entropy} there is a $\eps > 0$ such that 
$h_{\Sigma ,\mu}^\gamma (\beta) \leq h_{\Sigma ,\mu}^{\gamma,\eps} (\beta) + \kappa$ for all finite partitions
$\beta\subseteq\cT$ which refine $\gamma$.

Since $\cS$ is generating there are a finite partition $\xi\subseteq\cS$ and 
a nonempty finite set $K\subseteq G$ such that for every $B\in\gamma$ there is a $\Lambda_B \subseteq\xi^K$ 
for which the set $B' = \bigcup_{Y\in\Lambda_B} \bigcap_{s\in K} sY_s$ satisfies 
$\mu (B\Delta B') < \eps/16$. 

Take a finite partition $\alpha\subseteq\cS$ with $\alpha\geq\xi$, a finite set $F\subseteq G$ containing $K\cup \{ e \}$, 
and a $\delta > 0$ such that
\[
\limsup_{i\to\infty} \frac{1}{d_i} \log |\Hom_\mu (\alpha,F,\delta,\sigma_i)|_\xi \leq h_{\Sigma ,\mu}^\xi (\cS)
+ \kappa .
\]
By shrinking $\delta$ if necessary we may assume that it is less than $\eps/(8|\xi^K||K|)$.
Since $\cT$ is generating, there are a finite partition $\beta\subseteq\cT$ refining $\gamma$ and
a nonempty finite set $E\subseteq G$ such that for every $A\in\Sigma(\alpha_F)$ there is a 
$\Lambda_A \subseteq\beta^E$ for which the set 
$A' = \bigcup_{Y\in\Lambda_A} \bigcap_{s\in E} sY_s$ satisfies 
$\mu (A\Delta A') < \delta/(12|\alpha^F|)$.

Take a $\delta' > 0$ which is smaller than $\delta/(9|\alpha^F||\beta^E||E|)$ and also
small enough so that every $\varphi\in\Hom_\mu (\beta,FE,\delta',\sigma)$
satisfies $\zeta (\varphi (B)) \leq 2\mu (B)$ for all $B\in\Sigma (\beta_{FE})$.

By Lemma~\ref{L-hom} there is a homomorphism $\theta : \Sigma (\alpha_F) \to \Sigma (\beta_{FE})$ 
such that $\mu (\theta (A)\Delta A)$ is less than both $\delta/(12|\alpha^K|)$ and $\eps/(16|\xi^K|)$
for all $A\in\Sigma(\alpha_F)$. 

Let $\sigma$ be a map from $G$ to $\Sym (d)$ for some $d\in\Nb$ which we assume to be a good enough
sofic approximation to obtain an estimate below. Let $\varphi\in\Hom_\mu (\beta,FE,\delta',\sigma)$.
Set $\varphi^\natural = \varphi\circ\theta$.
We will show that $\varphi^\natural\in\Hom_\mu (\alpha,F,\delta,\sigma)$.

Let $t\in F$. Then for every $A\in\alpha$ we have, assuming that $\sigma$ is a good enough sofic approximation,
\begin{align*}
\frac1d |\varphi(tA')\Delta\sigma_t \varphi(A')|
&\leq \sum_{Y\in\Lambda_A} \sum_{s\in E} \frac1d |\varphi (tsY_s )\Delta \sigma_t \varphi (sY_s)| \\
&\leq \sum_{Y\in\Lambda_A} \sum_{s\in E} \frac1d \big( |\varphi (tsY_s )\Delta \sigma_{ts} \varphi (Y_s)|
+ |\sigma_{ts} \varphi (Y_s) \Delta \sigma_t \sigma_s \varphi (Y_s)| \\
&\hspace*{30mm} \ + |\sigma_t (\sigma_s \varphi (Y_s)\Delta \varphi (sY_s )| \big) \\
&< 3|\beta^E| |E| \delta' < \frac{\delta}{3|\alpha|} 
\end{align*}
and
$\mu (\theta(tA)\Delta tA' ) \leq \mu(\theta(tA)\Delta tA) + \mu (A\Delta A') < \delta/(6|\alpha|)$,
whence
\begin{align*}
\sum_{A\in\alpha} \frac1d |\varphi^\natural(tA)\Delta \sigma_t \varphi^\natural(A)|
&\leq \sum_{A\in\alpha} \frac1d \big( |\varphi (\theta(tA)\Delta tA'))|
+ |\varphi(tA')\Delta\sigma_t \varphi(A')| \\
&\hspace*{25mm} \ + |\sigma_t \varphi(A'\Delta\theta(A))|\big) \\
&< 2|\alpha|\mu(\theta(tA)\Delta tA') + \frac{\delta}{3} + 2|\alpha|\mu(A'\Delta\theta(A))
< \delta .
\end{align*}
Also, for every $A\in\alpha_F$ we have $| \zeta (\varphi (A')) - \mu (A') | < \delta' < \delta/(3|\alpha^F|)$ 
and $\mu (\theta(A)\Delta A' ) \leq \mu(\theta(A)\Delta A) + \mu (A\Delta A') < \delta/6|\alpha^F|$
and so
\begin{align*}
\sum_{A\in\alpha_F} | \zeta (\varphi^\natural(A)) - \mu (A) | 
&\leq \sum_{A\in\alpha_F} \big( | \zeta (\varphi(\theta(A) \Delta A')) | 
+ | \zeta (\varphi (A')) - \mu (A') | \\
&\hspace*{25mm} \ + | \mu (A' \Delta A) | \big) \\
&\leq \sum_{A\in\alpha_F} \bigg( 2 \mu(\theta(A) \Delta A') + 2\cdot\frac{\delta}{3|\alpha^F|} \bigg) < \delta .
\end{align*}
Thus $\varphi^\natural\in\Hom_\mu (\alpha,F,\delta,\sigma)$.

Let $\Gamma : \Hom_\mu (\beta,FE,\delta',\sigma) \to \Hom_\mu (\alpha,F,\delta,\sigma)$ be the map 
$\varphi\mapsto\varphi^\natural$. Take an $\eps' > 0$ such that $\eps' < \eps /(8|\xi^K||K|)$. 
Let $\varphi$ and $\psi$ be elements of $\Hom_\mu (\beta,FE,\delta',\sigma)$ with
$\rho_\xi (\varphi^\natural,\psi^\natural) < 2\eps'$. For every $B\in\gamma$ we have
\begin{align*}
|\varphi^\natural(B')\Delta\psi^\natural(B')| 
&\leq \sum_{Y\in\Lambda_B} \sum_{s\in K} |\varphi^\natural(sY_s)\Delta\psi^\natural(sY_s)| \\
&\leq \sum_{Y\in\Lambda_B} \sum_{s\in K} \big( |\varphi^\natural(sY_s)\Delta \sigma_s\varphi^\natural(Y_s)|
+ |\sigma_s (\varphi^\natural(Y_s)\Delta\psi^\natural(Y_s))| \\
&\hspace*{30mm} \ + |\sigma_s\psi^\natural(Y_s)\Delta \psi^\natural(sY_s)| \big) \\
&< |\xi^K||K|(\delta + 2\eps' + \delta ) < \frac{\eps}{2}
\end{align*}
and
\begin{align*}
\mu (B\Delta\theta(B')) 
&\leq \mu (B\Delta B') + \mu (B' \Delta\theta(B')) \\
&< \frac{\eps}{16} 
+ \sum_{Y\in\Lambda_B} \mu \big(\big({\textstyle\bigcap_{s\in K} sY} \big)\Delta\theta\big({\textstyle\bigcap_{s\in K} sY} \big)\big) \\
&< \frac{\eps}{16}  + |\Lambda_B | \cdot \frac{\eps}{16|\xi^K|}
\leq \frac{\eps}{8} 
\end{align*}
and hence
\begin{align*}
\rho_\gamma (\varphi,\psi) &= \max_{B\in\gamma} \frac1d |\varphi(B)\Delta\psi(B)| \\
&\leq \max_{B\in\gamma} \frac1d \big( |\varphi(B\Delta\theta(B'))|
+ |\varphi^\natural(B')\Delta\psi^\natural(B')|
+ |\psi(\theta(B')\Delta B)| \big) \\
&\leq 4\max_{B\in\gamma} \mu (B\Delta\theta(B')) + \frac{\eps}{2} < \eps .
\end{align*}
It follows that for every $(\eps,\rho_\gamma)$-separated set $Q\subseteq\Hom_\mu (\beta,FE,\delta',\sigma)$ 
the image $\Gamma(Q)$ is $(\eps',\rho_\xi)$-separated, and so
\begin{align*}
N_\eps (\Hom_\mu (\beta,FE,\delta',\sigma),\rho_\gamma) 
\leq N_{\eps'} (\Hom_\mu (\alpha,F,\delta,\sigma),\rho_\xi ) \leq |\Hom_\mu (\alpha,F,\delta,\sigma)|_\xi .
\end{align*}
Consequently,
\begin{align*}
h_{\Sigma ,\mu}^\gamma (\cT) \leq h_{\Sigma ,\mu}^\gamma (\beta) 
&\leq h_{\Sigma ,\mu}^{\gamma,\eps} (\beta) + \kappa \\
&\leq \limsup_{i\to\infty} \frac{1}{d_i} \log N_\eps (\Hom_\mu (\beta,FE,\delta',\sigma_i ),\rho_\gamma) + \kappa \\
&\leq \limsup_{i\to\infty} \frac{1}{d_i} \log |\Hom_\mu (\alpha,F,\delta,\sigma_i)|_\xi + \kappa \\
&\leq h_{\Sigma ,\mu}^\xi (\cS) + 2\kappa \leq h_{\Sigma ,\mu} (\cS) + 2\kappa .
\end{align*}
Since $\gamma$ was an arbitrary finite partition in $\cT$ and $\kappa$ an arbitrary positive number, 
we conclude that $h_{\Sigma ,\mu} (\cT) \leq h_{\Sigma ,\mu} (\cS)$, as desired.
\end{proof}

\section{Comparison with prior definitions}\label{S-comparison}

In this section we continue the notational conventions of the previous section, with the additional
assumption that the probability space $X$ is standard.
Our aim is to prove that $h_{\Sigma ,\mu} (X,G)$ agrees with the 
entropy defined in Section~2 of \cite{KerLi11a}. By Section~3 of \cite{KerLi11a}, this will also show
that $h_{\Sigma ,\mu} (X,G)$ agrees with Bowen's entropy 
in the presence of a generating partition with finite Shannon entropy. 

Write $h_{\Sigma,\mu}' (X,G)$ for the sofic measure entropy as defined in Section~2 of \cite{KerLi11}.
We will use the following equivalent formulation of $h_{\Sigma,\mu}' (X,G)$ in terms of topological models
(see Section~3 of \cite{KerLi11a}).
Suppose that $X$ is a compact metrizable space, the action of $G$ on $X$ is by homeomorphisms,
and $\rho$ is a compatible metric on $X$. 
For a given $d\in\Nb$, we define on the set of all maps from $\{ 1,\dots ,d\}$ to $X$ the pseudometrics
\begin{align*}
\rho_2 (\varphi ,\psi ) &= \bigg( \frac1d \sum_{k=1}^d \rho (\varphi (k),\psi (k))^2 \bigg)^{1/2} , \\
\rho_\infty (\varphi ,\psi ) &= \max_{k=1,\dots ,d} \,\rho (\varphi (k),\psi (k)) .
\end{align*}
Let $F$ be a nonempty finite subset of $G$, $L$ a finite subset of $C(X)$, and $\delta > 0$.
Let $\sigma$ be a map from $G$ to $\Sym (d)$ for some $d\in\Nb$.
We write $\Map_\mu (\rho ,F, L,\delta ,\sigma )$ for the set of all maps 
$\varphi : \{ 1,\dots ,d\} \to X$ such that, writing $\iota$ for the action of $G$ on $X$,
\begin{enumerate}
\item $\rho_2 (\varphi\circ\sigma_s , \iota_s \circ\varphi ) < \delta$ for all $s\in F$, and

\item $\big| (\varphi_*\zeta)(f) - \mu (f) \big| < \delta$ for all $f\in L$.
\end{enumerate}
For $\varepsilon > 0$ we define
\begin{align*}
h_{\Sigma ,\mu}^\varepsilon (\rho ,F, L,\delta ) &=
\limsup_{i\to\infty} \frac{1}{d_i} \log N_\varepsilon (\Map_\mu (\rho ,F, L,\delta ,\sigma_i ),\rho_\infty ) ,\\
h_{\Sigma ,\mu}^\varepsilon (\rho ) &= \inf_{F} \inf_{L} \inf_{\delta > 0} h_{\Sigma ,\mu}^\varepsilon (\rho ,F,L,\delta) ,\\
h_{\Sigma ,\mu} (\rho ) &= \sup_{\varepsilon > 0} h_{\Sigma ,\mu}^\varepsilon (\rho ) ,
\end{align*}
where in the second line $L$ ranges over the finite subsets of $C(X)$ and
$F$ ranges over the nonempty finite subsets of $G$. 
By Proposition~3.4 of \cite{KerLi11a} we have $h_{\Sigma,\mu}' (X,G) = h_{\Sigma ,\mu} (\rho )$,
a fact which will be used tacitly in the proof below.

\begin{theorem}
$h_{\Sigma,\mu} (X,G) = h_{\Sigma,\mu}' (X,G)$.
\end{theorem}

\begin{proof}
We begin by showing that $h_{\Sigma,\mu} (X,G) \leq h_{\Sigma,\mu}' (X,G)$. 
Suppose first that we are in the case $h_{\Sigma,\mu} (X,G) < \infty$. Then given a $\kappa > 0$ we can find
a finite measurable partition $\xi$ of $X$ such that 
$h_{\Sigma,\mu} (X,G) \leq h_{\Sigma,\mu}^\xi (\cB) + \kappa$. 
By replacing $X$ with the spectrum
of some separable unital $G$-invariant C$^*$-subalgebra of $L^\infty(X,\mu)$ containing
the characteristic functions of the atoms of $\xi$, we may assume that $X$ is a compact
metrizable space with compatible metric $\rho$, 
the action of $G$ is by homeomorphisms, and $\xi$ is a clopen partition of $X$. 

Let $\eps > 0$ be smaller than the minimum of the
Hausdorff distances between $B$ and $B'$ over all distinct $B,B' \in\xi$.
Let $F$ be a finite symmetric subset of $G$ containing $e$ and $L$ a finite subset of $C(X)$. 
Let $\delta > 0$, to be further specified.
Let $\delta' > 0$, to be further specified as a function of $\delta$, $n$, $|\alpha_F|$, and $L$.
Let $\sigma : G\to\Sym(d)$ be a good enough sofic approximation for a purpose to be described shortly.

Let $\alpha = \{ A_1 , \dots , A_n \}$ be a finite measurable partition of $X$ refining $\xi$ such that
each $A_i$ has diameter less than $\delta$.
Given a $\varphi\in\Hom_\mu (\alpha,F,\delta',\sigma)$, we construct a map 
$\hat{\varphi} : \{ 1,\dots,d \} \to X$ by considering for
each $k\in \{ 1,\dots,d \}$ the element $f\in\{1,\dots,n\}^F$ such that
$k\in\varphi (\bigcap_{t\in F} tA_{f(t)})$ and defining $\hat{\varphi} (k)$ to be any point in
$\bigcap_{t\in F} tA_{f(t)}$. Given an $s\in F$, set 
\[
C_s = \bigcup_{i=1}^n \big( \sigma_s^{-1} \varphi (A_i) \cap \varphi (s^{-1} A_i) \big) .
\]
Assuming $\sigma$ is a good enough sofic approximation so that the restrictions of $\sigma_s^{-1}$ 
and $\sigma_{s^{-1}}$ agree on a subset of proportional size sufficiently close to one, we will have 
$|C_s|/d \geq 1 - 2n\delta'$. Now for $k\in C_s$ and $i=1,\dots ,n$ we have,
writing $\Upsilon_{e,i}$ for the set of all $f\in \{ 1,\dots,n \}^F$ such that $f(e) = A_i$
and $\Upsilon_{s^{-1},i}$ for the set of all $f\in \{ 1,\dots,n \}^F$ such that $f(s^{-1}) = A_i$,
\begin{align*}
\hat{\varphi} (\sigma_s (k)) \in A_i = \bigcup_{f\in\Upsilon_{e,i}} \bigcap_{t\in F} tA_{f(t)}
&\Leftrightarrow \sigma_s (k) \in \bigcup_{f\in\Upsilon_{e,i}} 
\varphi \bigg( \bigcap_{t\in F} tA_{f(t)}\bigg) = \varphi (A_i ) \\
&\Leftrightarrow k \in\sigma_s^{-1} \varphi (A_i ) \\
&\Leftrightarrow k \in\varphi (s^{-1} A_i ) 
= \bigcup_{f\in\Upsilon_{i,s^{-1}}} \varphi \bigg( \bigcap_{t\in F} tA_{f(t)} \bigg) \\
&\Leftrightarrow \hat{\varphi} (k) 
\in \bigcup_{f\in\Upsilon_{s^{-1},i}} \bigcap_{t\in F} tA_{f(t)} = s^{-1} A_i \\
&\Leftrightarrow s\hat{\varphi} (k) \in A_i .
\end{align*}
Since the diameter of $A_i$ is less than $\delta'$ this shows that 
$\rho (\hat{\varphi} (\sigma_s (k)),s\hat{\varphi} (k)) < \delta'$ and hence, if $\delta'$ is small enough
as a function of $\delta$ and $n$, that $\rho_2 (\hat{\varphi} \circ\sigma_s ,\iota_s \circ\hat{\varphi} ) < \delta$
for all $s\in F$ where $\iota$ denotes the action of $G$ on $X$. Also, for $A\in\alpha_F$ we have
$\hat{\varphi}_* \zeta (1_A) = \zeta (\hat{\varphi}^{-1} (A)) = \zeta(\varphi(A))$, and thus, for $f\in L$
if we find scalars $c_{f,A}$ such that 
$\| f - \sum_{A\in\alpha_F} c_{f,A} 1_A \|_\infty < \delta /3$ and take $\delta'$ small
enough so that $\sum_{A\in\alpha_F} |c_{f,A}| |\zeta(\varphi(A)) - \mu(A)| < \delta /3$ we will have
\begin{align*}
|(\hat{\varphi}_*\zeta)(f) - \mu (f)| 
&\leq \bigg|(\hat{\varphi}_*\zeta)\bigg( f -  \sum_{A\in\alpha_F} c_{f,A} 1_A\bigg)\bigg|
+ \sum_{A\in\alpha_F} |c_{f,A}| |\zeta(\varphi(A)) - \mu(A)| \\
&\hspace*{25mm} \ + \bigg|\mu \bigg(\sum_{A\in\alpha_F} c_{f,A} 1_A - f \bigg) \bigg| < \delta .
\end{align*}
Consequently $\hat{\varphi} \in\Map_\mu (\rho,F,L,\delta,\sigma)$.
Moreover, if $\varphi$ and $\psi$ are elements of $\Hom_\mu (\alpha,F,\delta',\sigma)$
whose restrictions to $\xi$ differ, then $\rho_\infty (\hat{\varphi},\hat{\psi}) > \eps$ by our choice of $\eps$,
and so $|\Hom_\mu (\alpha,F,\delta',\sigma)|_\xi \leq N_\eps (\Map_\mu (\rho,F,L,\delta,\sigma),\rho_\infty )$.
Hence $h_{\Sigma,\mu}^\xi (\alpha,F,\delta') \leq h_{\Sigma,\mu}^\eps (\rho,F,L,\delta)$ 
and so
\[
h_{\Sigma,\mu} (X,G) - \kappa \leq h_{\Sigma,\mu}^\xi (\cB) \leq h_{\Sigma,\mu}^\eps (\rho) 
\leq h_{\Sigma,\mu} (\rho) = h_{\Sigma,\mu}' (X,G) .
\]
Since $\kappa$ was an arbitrary positive number we thus obtain $h_{\Sigma,\mu} (X,G) \leq h_{\Sigma,\mu}' (X,G)$.
In the case $h_{\Sigma,\mu} (X,G) = \infty$ we can argue in the same way only starting with the fact 
that for every $M>0$ there is a finite measurable partition $\xi$ of $X$ satisfying $h_{\Sigma,\mu}^\xi (\cB) \geq M$.

We now establish the reverse inequality.
By replacing $X$ with the spectrum of the unital $G$-invariant C$^*$-subalgebra of $L^\infty(X,\mu)$
generated by the characteristic functions of some countable generating collection of measurable sets,
we may assume that $X$ is a zero-dimensional compact metrizable space with compatible metric $\rho$ and that
the action of $G$ is by homeomorphisms. Suppose that we are in the case $h_{\Sigma,\mu} (\rho) < \infty$. Then
given a $\kappa > 0$ there exists an $\eps > 0$ such that
$h_{\Sigma,\mu} (\rho) \leq h_{\Sigma,\mu}^\eps (\rho) + \kappa$. Pick a finite measurable partition
$\xi$ of $X$ such that the diameter of each of its atoms is less than $\eps$.
Take a finite measurable partition $\alpha = \{ A_1 , \dots , A_n \}$ of $X$ refining $\xi$, a
finite symmetric subset $F$ of $G$ containing $e$, and a $\delta > 0$ such that
$h_{\Sigma,\mu}^\xi (\alpha,F,\delta) \leq h_{\Sigma,\mu}^\xi (\cB) + \kappa$.
Now it is readily seen that if $\alpha' = \{ A_1' , \dots , A_n' \}$ is an ordered measurable partition of $X$ refining $\xi$
such that $\max_{i=1,\dots,n} \mu (A_i \Delta A_i' )$ is sufficiently small then, for
each sufficiently good sofic approximation $\sigma : G\to\Sym(d)$, 
composing an element of $\Hom_\mu (\alpha',F,\delta/2,\sigma)$ with 
the homomorphism $\theta : \Sigma (\alpha_F) \to \Sigma (\beta_F)$ defined by 
$\theta (\bigcap_{t\in F} tA_{f(t)} ) = \bigcap_{t\in F} tB_{f(t)}$ for all $f\in\{1,\dots,n\}^F$
yields an element of $\Hom_\mu (\alpha,F,\delta,\sigma)$, 
in which case $|\Hom_\mu (\alpha',F,\delta/2,\sigma)|_\xi \leq |\Hom_\mu (\alpha,F,\delta,\sigma)|_\xi$.
By a standard approximation argument one can find such an $\alpha'$ consisting
of clopen sets, and so we may assume that $\alpha$ itself is a clopen partition of $X$.

Let $\delta'$ be a positive number to be determined
which is smaller than the Hausdorff distance between $A_i$ and $A_j$ 
for all distinct $i,j=1,\dots,n$. 
Write $L$ for the set of characteristic functions of the atoms of $\alpha_F$.
Let $\sigma$ be a map from $G$ to $\Sym (d)$ for some $d\in\Nb$.
Given a $\varphi\in\Map_\mu (\rho,F,L,(|F|n)^{-1}\delta,\sigma)$, we construct a homomorphism
$\hat{\varphi} : \Sigma(\alpha_F) \to \cP_d$ by declaring $\hat{\varphi} (\bigcap_{t\in F} tA_{f (t)} )$ 
for a $f\in \{ 1,\dots ,n\}^F$ to be the set of all $k\in\{ 1,\dots,d\}$ such that
$\varphi (k) \in\bigcap_{t\in F} tA_{f (t)}$. For $s\in F$
write $D_s$ for the set of all $k\in\{ 1,\dots,d \}$ such that $\rho (\varphi (\sigma_s^{-1} (k)) , s^{-1} \varphi (k)) < \delta'$.
Assuming $\sigma$ is a good enough sofic approximation so that the restrictions of $\sigma_s^{-1}$ 
and $\sigma_{s^{-1}}$ agree on a subset of proportional size sufficiently close to one, 
and that $\delta'$ is sufficiently small as a function of $(|F|n)^{-1}\delta$, we have $|D_s |/d \geq 1-\delta/n$.
Also, for $k\in D_s$ and $i=1,\dots,n$ we have, 
writing $\Upsilon_{e,i}$ for the set of all $f\in \{ 1,\dots,n \}^F$ such that $f(e) = A_i$
and $\Upsilon_{s,i}$ for the set of all $f\in \{ 1,\dots,n \}^F$ such that $f(s) = A_i$,
\begin{align*}
k\in\sigma_s \hat{\varphi} (A_i) 
&\Leftrightarrow \sigma_s^{-1} (k) \in\hat{\varphi} (A_i) 
= \bigcup_{f\in\Upsilon_{e,i}} \hat{\varphi} \bigg(\bigcap_{t\in F} tA_{f(t)} \bigg) \\
&\Leftrightarrow \varphi (\sigma_s^{-1} (k)) 
\in\bigcup_{f\in\Upsilon_{e,i}} \bigcap_{t\in F} tA_{f(t)} = A_i \\
&\Leftrightarrow \varphi(k) \in sA_i = \bigcup_{f\in\Upsilon_{s,i}} \bigcap_{t\in F} tA_{f(t)} 
\hspace*{3mm} \text{(by our choice of $\delta'$)}\\
&\Leftrightarrow k\in \bigcup_{f\in\Upsilon_{s,i}} \hat{\varphi} \bigg( \bigcap_{t\in F} tA_{f(t)} \bigg)
= \hat{\varphi} (sA_i ) .
\end{align*}
Consequently $\sum_{i=1}^n |\sigma_s \hat{\varphi} (A_i) \Delta \hat{\varphi} (sA_i)| < \delta$
for all $s\in F$ assuming $\delta'$ is small enough.
Since $\zeta (\hat{\varphi} (A)) = (\varphi_* \zeta)(1_A)$ for every $A\in \alpha_F$
so that $\sum_{A\in\alpha_F} |\zeta (\hat{\varphi} (A)) - \mu(A)| < |\alpha_F| (|F|n)^{-1} \delta\leq\delta$, it follows that
$\hat{\varphi} \in\Hom_\mu (\alpha,F,\delta,\sigma)$.

Now since the atoms of $\xi$ all have diameter less than $\eps$, we have $\hat{\varphi} |_\xi \neq \hat{\psi} |_\xi$ for all
$\varphi,\psi\in \Map_\mu (\rho,F,L,\delta,\sigma)$ satisfying $\rho_\infty (\varphi,\psi) > \eps$.
Therefore 
\[
N_\eps (\Map_\mu (\rho,F,L,(|F|n)^{-1}\delta ,\sigma),\rho_\infty ) \leq |\Hom_\mu (\alpha,F,\delta,\sigma)|_\xi 
\]
and hence $h_{\Sigma,\mu}^\eps (\rho,F,L,(|F|n)^{-1}\delta) \leq h_{\Sigma,\mu}^\xi (\alpha,F,\delta)$.
We thus deduce that
\[
h_{\Sigma,\mu} (\rho) - \kappa \leq h_{\Sigma,\mu}^\eps (\rho) \leq h_{\Sigma,\mu}^\xi (\alpha) \leq 
h_{\Sigma,\mu} (X,G) + \kappa .
\]
Since $\kappa$ was an arbitrary positive number we conclude that 
$h_{\Sigma,\mu}' (X,G) \leq h_{\Sigma,\mu} (X,G)$. In the case $h_{\Sigma,\mu} (\rho) = \infty$ we can apply
the same argument only starting from the fact that for every $M>0$ there is an $\eps > 0$ for which 
$h_{\Sigma,\mu}^\eps (\rho) \geq M$.
\end{proof}

\section{Bernoulli actions}\label{S-Bernoulli}

Here we show how to compute the sofic entropy of Bernoulli actions according to Definition~\ref{D-sofic meas ent}. 
This only depends on the asymptotic freeness of the sofic approximation sequence, 
in contrast to Definition~\ref{D-sofic meas ent} and Theorem~\ref{T-generator}, 
which do not depend on it at all. The following lemma will 
permit us to reduce the computation to Bowen's arguments in the finite base case \cite{Bow10}.
As in the previous sections, $\Sigma = \{ \sigma_i : G \to \Sym (d_i ) \}_{i=1}^\infty$ is an
arbitrary fixed sofic approximation sequence.

\begin{lemma}\label{L-partition ent ineq}
Let $(X,\mu)$ be a probability space and $G\curvearrowright X$ a measure-preserving action.
Let $\xi$ and $\alpha$ be finite measurable partitions of $X$ with $\alpha\geq\xi$. Then
\begin{enumerate}
\item $h_{\Sigma,\mu}^\xi (\alpha)\leq H_\mu (\xi)$,

\item $h_{\Sigma,\mu}^\xi (\alpha) \geq h_{\Sigma,\mu}^\alpha (\alpha) - H_\mu (\alpha|\xi)$.
\end{enumerate}
\end{lemma}

\begin{proof}
First we prove (i). 
Write $\xi = \{ B_1 , \dots , B_n \}$. Let $\eps > 0$. 
By the continuity properties of $H(\cdot)$, there is a $\delta > 0$ such that, for all $d\in\Nb$,
if $\gamma = \{ C_1 , \dots , C_n \}$ is an ordered partition of $\{ 1,\dots , d\}$ with 
$\sum_{i=1}^n \big| \mu (B_i) - \zeta (C_i) \big| < \delta$ then $|H_\mu (\xi) - H_\zeta (\gamma)| < \eps$.

Fix a $d\in\Nb$. Write $T$ for the set of tuples $(c_1,\dots,c_n)\in \{ 1/d,2/d\dots ,1\}^n$
such that $\sum_{i=1}^n c_i = 1$ and $\sum_{i=1}^n \big| \mu (B_i) - c_i \big| < \delta$. 
For each $c=(c_1,\dots,c_n)\in T$ write $W_c$ for the set of all ordered partitions $\gamma = \{ C_1 , \dots , C_n \}$ 
of $\{ 1,\dots , d\}$ such that $|C_i|/d = c_i$ for every $i=1,\dots,n$. 
By our choice of $\delta$ and Stirling's approximation, we have, assuming $d$ is sufficiently large,
\[
|W_c| = \frac{d!}{(c_1 d)! \cdots (c_n d)!}\leq 
\prod_{i=1}^n c_i^{-c_i d(1+\eps)}\leq e^{d(1+\eps)(H_\mu (\xi) + \eps)}
\]
for every $c\in T$.
Note also that $|T|\leq (2\delta d)^n$ since there are at most $2\delta d$ choices for the value of each 
$c_i$ among the elements of $T$. The total number of homomorphisms $\varphi : \Sigma(\xi)\to\cP_d$
satisfying $\sum_{i=1}^n |\zeta(\varphi(B_i) - \mu(B_i)| < \delta$ is equal to $|\bigcup_{c\in T} W_c |$, 
and from what we have just observed this is bounded above by 
$(2\delta d)^n e^{d(1+\eps)(H_\mu (\xi) + \eps)}$. Hence  
\begin{align*}
h_{\Sigma,\mu}^\xi (\alpha) \leq h_{\Sigma,\mu}^\xi (\alpha,\{ e \},\delta) 
\leq h_{\Sigma,\mu}^\xi (\xi,\{ e \},\delta) \leq (1+\eps)(H_\mu (\xi) + \eps) . 
\end{align*}
Since $\eps$ was an arbitrary positive number we obtain (i).

Now let us prove (ii). Let $\eps > 0$. 
Write $\xi = \{ B_1,\dots,B_m\}$.
For each $i=1,\dots ,m$ write $\alpha_i = \{ A_{i,1} , \dots ,A_{i,n_i}\}$ for the partition of $B_i$ consisting
of those members of $\alpha$ contained in $B_i$. By the continuity properties of $H(\cdot)$
there is a $\delta > 0$ such that, for all $d\in\Nb$,
if $\psi : \Sigma(\alpha)\to\cP_d$ is a homomorphism satisfying
$\sum_{A\in\alpha} \big| \zeta(\psi(A)) - \mu (A) \big| < \delta$ 
then $|H_{\mu_i} (\alpha_i) - H_{\zeta_i} (\psi(\alpha_i))| < \eps$ for every $i=1,\dots,m$,
where $\mu_i$ is $\mu(B_i)^{-1}$ times the restriction of $\mu$ to $B_i$ and
$\zeta_i$ is $\zeta(\psi(B_i))^{-1}$ times the restriction of $\zeta$ to $\psi(B_i)$.

Suppose we are given a homomorphism $\psi : \Sigma(\alpha)\to\cP_d$ satisfying 
$\sum_{A\in\alpha} \big| \zeta(\psi(A)) - \mu (A) \big| < \delta$.
Write $Q_i$ for the set of all ordered partitions $\{ C_1 ,\dots ,C_{n_i} \}$ of $\psi (B_i)$
satisfying $\sum_{j=1}^{n_i} \big| \mu (A_{i,j}) - \zeta (C_j) \big| < \delta$. Then the
set of all homomorphisms $\varphi : \Sigma(\alpha)\to\cP_d$ which satisfy 
$\sum_{i=1}^m \big| \zeta(\varphi(B_i)) - \mu (B_i) \big| < \delta$ and restrict to $\psi$ on $\xi$ 
has cardinality at most $\prod_{i=1}^m |Q_i|$. 
By an estimate as in the second paragraph using Stirling's approximation, assuming $d$ is large enough
the set $Q_i$ has cardinality at most 
$(2\delta d)^{n_i} e^{\zeta(\psi(B_i))(1+\eps)(H_{\mu_i}(\alpha_i) + \eps)}$, and since 
\[
\sum_{i=1}^m \zeta(\psi(B_i))H_{\mu_i} (\alpha_i) \leq \sum_{i=1}^m \mu(B_i) H_{\mu_i} (\alpha_i) + \delta
= H_\mu (\alpha|\xi) + \delta
\]
this gives
\[
\prod_{i=1}^m |Q_i| \leq (2\delta d)^m e^{(1+\eps)(H_\mu (\alpha|\xi) + \delta + \eps)} .
\]
It follows that for every nonempty finite set $F\subseteq G$ we have
\begin{align*}
h_{\Sigma,\mu}^\xi (\alpha,F,\delta) \geq h_{\Sigma,\mu}^\alpha (\alpha,F,\delta) 
- (1+\eps)(H_\mu (\alpha | \xi) + \delta + \eps) 
\end{align*}
and hence 
$h_{\Sigma,\mu}^\xi (\alpha,F)\geq h_{\Sigma,\mu}^\alpha (\alpha) - (1+\eps)(H_\mu (\alpha | \xi) + \eps)$
as $\delta$ can be taken arbitrarily small.
Since $\eps$ was an arbitrary positive number, this yields (ii).
\end{proof}

For a probability space $(Y,\nu)$ we define $H(\nu)$ to be the supremum of $H_\nu (\alpha)$
over all finite measurable partitions $\alpha$ of $Y$.

\begin{theorem}
Let $(Y,\nu)$ be a probability space and let $G\curvearrowright (Y^G,\nu^G)$ be the associated Bernoulli action.
Then $h_{\Sigma,\nu^G} (Y^G,G) = H(\nu)$.
\end{theorem}

\begin{proof}
Let $\cS$ be the $\sigma$-algebra consisting of those measurable subsets of $Y^G$ which are cylinder sets over $e$.
Then $\cS$ is generating and so it suffices by Theorem~\ref{T-generator} to show that 
$h_{\Sigma,\mu} (\cS) = H(\nu)$.
By Lemma~\ref{L-partition ent ineq}(i) we have $h_{\Sigma,\mu} (\cS)\leq H(\nu)$, and so
we concentrate on the reverse inequality. 

Let $\xi$ and $\alpha$ be finite partitions in $\cS$ with $\alpha\geq\xi$. We will show that 
$h_{\Sigma,\mu}^\xi (\alpha) \geq H_\mu (\xi)$, from which the desired inequality 
$h_{\Sigma,\mu} (\cS)\geq H(\nu)$ ensues.
By Lemma~\ref{L-partition ent ineq}(ii) we need only prove that $h_{\Sigma,\mu}^\alpha (\alpha) \geq H_\mu (\alpha)$.
This is essentially contained in Theorem~8.1 in \cite{Bow10}, and we will reproduce the argument from there.

Let $\delta > 0$. Let $\eta > 0$ be such that $2\eta < |\alpha|^{-|F|} \delta$.
Let $F$ be a finite subset of $G$ containing $e$.
Let $d\in\Nb$ and let $\sigma$ be a map from $G$ to $\Sym(d)$ which is a sufficiently good
sofic approximation to satisfy a couple of conditions to be specified below. 
Write $V$ for the set
of all $v\in\{1,\dots ,d\}$ such that $\sigma_s \sigma_t (v) = \sigma_{st} (v)$ for all $s,t\in F$ 
and $\sigma_s^{-1} (v) \neq \sigma_t^{-1} (v)$ for all distinct $s,t\in F$.

Enumerate the elements of $\alpha$ as $A_1 , \dots ,A_n$. 
Write $\kappa$ for the probability measure on $\{ 1,\dots ,n\}$ determined by $\kappa (\{ i \}) = \mu (A_i)$.
We view $\{ 1,\dots ,n\}^d$ as a probability space with the product measure $\kappa^d$.

Fix an $f\in \{ 1,\dots,n \}^F$. Given a $\gamma\in \{1,\dots,n\}^d$ we can think of it as
an ordered partition $\{\gamma^{-1} (1),\dots ,\gamma^{-1} (n)\}$ of $\{ 1,\dots,d \}$, and in accord with this viewpoint
we set $Q_{\gamma,f} = \bigcap_{s\in F} \sigma_s \gamma^{-1} (f(s))$.
Write $P_f$ for $\bigcap_{s\in F} sA_{f(s)}$. For every $\gamma\in \{ 1,\dots ,n\}^d$
denote by $\varphi_\gamma$ the homomorphism $\Sigma(\alpha_F) \to\cP_d$ determined by
$\varphi_\gamma (P_f) = Q_{\gamma,f}$. Note that for $s\in F$ and $i=1,\dots,n$ we have,
writing $\Upsilon_{s,i}$ for the set of $f\in \{1,\dots,n\}^F$ such that $f(s)=i$,
\begin{align*}
sA_i = sA_i \cap X = sA_i \cap \bigg( \bigcap_{t\in F\setminus\{s\}} \bigsqcup_{j=1}^n tA_j \bigg)
= \bigsqcup_{f\in\Upsilon_{s,i}} P_f
\end{align*}
and similarly $\sigma_s \gamma^{-1} (i) = \bigsqcup_{f\in\Upsilon_{s,i}} Q_{\gamma,f}$ so that
\begin{align*}
|\varphi_\gamma (sA_i) \Delta \sigma_s \varphi_\gamma (A_i) |
\leq |\sigma_s (\gamma^{-1} (i) \Delta \sigma_e \gamma^{-1} (i))| 
\leq |\{ v\in \{ 1,\dots,d\} : \sigma_e (v) \neq v \}| < \delta d 
\end{align*}
assuming $\sigma$ is a good enough sofic appproximation.
Thus we get that condition (i) in the definition of $\Hom_\mu (\alpha,F,\delta,\sigma)$ is
satisfied by $\varphi_\gamma$ for all $\gamma$. We now aim to get a lower bound on the number of $\gamma$ for which
$\varphi_\gamma$ satisfies condition (ii) in the same definition and hence lies in $\Hom_\mu (\alpha,F,\delta,\sigma)$.

For $v\in\{1,\dots,d\}$ we let $Z_v = Z_{v,f}$ be the function (random variable) on $\{ 1,\dots ,n\}^d$ 
which at a point $\gamma$ takes the value $1$ 
if $v\in V\cap Q_{\gamma,f}$ and $0$ otherwise.

Write $\Eb (\cdot)$ for the expected value of a function on $\{ 1,\dots,n \}^d$, 
that is, the integral with respect to $\kappa^d$. For $v\notin V$ we have $\Eb (Z_v) = 0$.
For $v\in V$, since $\sigma_s^{-1} (v) \neq \sigma_s^{-1} (v)$ for distinct $s,t\in V$ we have 
\begin{align*}
\Eb (Z_v) &= \kappa^d \big(\big\{ \gamma\in \{1,\dots,n\}^d : \sigma_s^{-1} (v) \in \gamma^{-1} (f(s))
\text{ for every } s\in F \big\}\big) \\
&= \prod_{s\in F} \kappa (\{ f(s) \}) = \prod_{s\in F} \nu(sA_{f(s)}) 
= \nu^G (P_f) .
\end{align*}
Set $Z = \sum_{v=1}^d Z_v$. We will estimate the variance $\Var (Z)$.
Let $v,w\in \{ 1,\dots ,d\}$. If one of $v$ and $w$ is not in $V$ then $Z_v Z_w = 0$ 
and so $\Eb (Z_v Z_w) = 0$. 
If $\sigma_s (v) \neq \sigma_t (w)$ for all $s,t\in F$ then
$Z_v$ and $Z_w$ are independent, i.e., $\Eb (Z_v Z_w) = \Eb (Z_v) \Eb (Z_w)$.
Thus the number of pairs $(v,w)\in V\times V$ for which $Z_v$ and $Z_w$ are not 
independent is at most $|V||F|^2$, which is bounded above by $d|F|^2$. Hence
\begin{align*}
\Eb (Z^2 ) = \sum_{v,w=1}^d \Eb (Z_v Z_w) 
\leq \sum_{v,w=1}^d \Eb (Z_v) \Eb (Z_w) + d|F|^2
= \Eb (Z)^2 + d|F|^2 . 
\end{align*}
and so $\Var (Z) = \Eb (Z^2) - \Eb (Z)^2 \leq d|F|^2$. Chebyshev's inequality then yields, for all $t>0$,
\begin{align*}
\Prob \big( |Z/d - \Eb (Z)/d | > t \big) \leq \frac{\Var (Z)}{d^2 t^2} \leq \frac{|F|^2}{dt^2} .
\end{align*}
Assuming that $\sigma$ is a good enough sofic approximation so that $\zeta (V) \geq 1-\eta$, for every $\gamma$ we have
\begin{align*}
\big|\zeta (Q_{\gamma,f}) - \nu^G (P_f) \big|
&\leq \big| \zeta (Q_{\gamma,f}) - \zeta (V\cap Q_{\gamma,f}) \big|
+ |(Z/d)(\gamma) - \Eb (Z)/d | \\
&\hspace*{20mm} \ + \big| \zeta (V)\nu^G (P_f) - \nu^G (P_f) \big| \\
&\leq |(Z/d)(\gamma) - \Eb (Z)/d | + 2\eta
\end{align*}
and thus, for $t>2\eta$,
\begin{align*}
\Prob \big( |\zeta (Q_{\gamma,f}) - \nu^G (P_f) | > t \big) 
\leq \frac{|F|^2}{d(t-2\eta)^2} .
\end{align*}
Taking $t = n^{-|F|} \delta$, which is larger than $2\eta$, we get, assuming $d$ is large enough,
\begin{align*}
\Prob \bigg( \big|\zeta (Q_{\gamma,f}) - \nu^G (P_f) \big| > \frac{\delta}{n^{|F|}}\bigg) \leq \frac{\delta}{n^{|F|}}
\end{align*}
Thus the probability that a $\gamma\in\{ 1,\dots ,n\}^d$ 
satisfies $|\zeta (Q_{\gamma,f}) - \nu^G (P_f) | > n^{-|F|}\delta$ for some $f\in \{ 1,\dots ,n\}^F$
is at most $\delta$. If this does not happen for a given $\gamma$ then
\begin{align*}
\sum_{f\in\{ 1,\dots,n \}^F} \big|\zeta(\varphi_\gamma (P_f)) - \nu^G (P_f) \big|
= \sum_{f\in\{ 1,\dots,n \}^F} \big|\zeta(Q_{\gamma,f}) - \nu^G (P_f) \big| < \delta .
\end{align*}
Hence the probability that $\gamma$ satisfies $\varphi_\gamma \in\Hom_\mu (\alpha,F,\delta,\sigma)$ is at least $1-\delta$.
Now we need to use this to estimate the actual number of $\gamma$ satisfying 
$\varphi_\gamma \in\Hom_\mu (\alpha,F,\delta,\sigma)$.

With $\gamma$ ranging as usual in the probability space $\{ 1,\dots,n \}^d$ we have,
by the law of large numbers and the independence of the coordinates of $\gamma$,
\[
\lim_{d\to\infty} \Prob \Big( \Big| -\frac1d \log \kappa^d (\gamma) - H(\kappa) \Big| > \delta \Big) = 0 .
\]
Thus, assuming $d$ is sufficiently large we can find an $L\subseteq \{ 1,\dots,n\}^d$ for which $\kappa^d (L) > 1-\delta$ 
and $\kappa^d (\{ \gamma \}) \leq e^{-d(H(\kappa)-\delta)}$
for all $\gamma\in L$. Then, writing $L_0$ for the set of all $\gamma\in L$ such that 
$\varphi_\gamma \in\Hom_\mu (\alpha,F,\delta,\sigma)$ we have $\kappa^d (L_0 ) \geq 1-2\delta$ and hence
\begin{align*}
|\Hom_\mu (\alpha,F,\delta,\sigma)|_\alpha \geq
|L_0| \geq \kappa^d (L_0) e^{d(H(\kappa)-\delta)} \geq (1-2\delta) e^{d(H(\kappa)-\delta)} .
\end{align*}
Since $\delta$ can be taken arbitrarily small,
it follows that $h_{\Sigma,\nu^G}^\alpha (\alpha,F) \geq H(\kappa) = H_\mu (\alpha)$.
Since $F$ was arbitrary finite subset of $G$ containing $e$, we conclude that
$h_{\Sigma,\nu^G}^\alpha (\alpha) \geq H_\mu (\alpha)$.
\end{proof}

\end{document}